\documentclass[11pt]{article} 
\usepackage[utf8]{inputenc}
\usepackage[english]{babel}
\usepackage{geometry}
\usepackage{authblk}
\usepackage{amssymb}
\usepackage{pgf}
\usepackage{mathbbol}
\usepackage[version=3]{mhchem}
\usepackage{tikz}
\usepackage{mathabx}
\usepackage{multimedia,color}
\usepackage{amsmath}
\usepackage{graphicx}
\usepackage{color}
\usepackage{amsthm}
\usepackage{algorithm}
\usepackage{algpseudocode}
\usepackage{caption}
\usepackage{subcaption}

\theoremstyle{plain}
\newtheorem{thm}{Theorem}[section]
\newtheorem{prop}{Proposition}[section]
\newtheorem{lem}{Lemma}[section]

\theoremstyle{definition}

\newtheorem{ass}{Assumption}[section]

\renewcommand{\P}{\mathbb{P}}
\newcommand{\E}{\mathbb{E}}
\newcommand{\N}{\mathbb{N}}
\newcommand{\R}{\mathbb{R}}
\newcommand{\err}{\eta}

\newcommand{\cK}{\mathcal{K}}
\newcommand{\cT}{\mathcal{T}}

\newcommand{\cN}{\mathcal{N}}

\newcounter{todocounter}
\newlength{\todowidthinner}
\usepackage{xifthen}
\usepackage{marginnote}
\DeclareRobustCommand{\MyChange}[3][\empty]{%
  {\color{#2}#3}
  \ifthenelse{\isempty{#1}}{}
  {%
    \addtocounter{todocounter}{1}%
    \ifmmode%
        {\color{#2}\text{$^{\framebox{\arabic{todocounter}}}$}}%
    \else%
        {\color{#2}\text{$^{\arabic{todocounter}}$}}%
    \fi%
    \marginpar{\textcolor{#2}{$^{\arabic{todocounter}}$\textnormal{#1}}}%
  }%
}
\definecolor{lightgray}{rgb}{0.7,0.7,0.7}
\usepackage[normalem]{ulem}

\newcommand{\MD}[2][]{\MyChange[#1]{black}{#2}}

\title{Adaptive Multilevel Monte Carlo Methods for Stochastic Variational Inequalities}
\author[1]{Ralf Kornhuber}
\author[1]{Evgenia Youett}
\affil[1]{Freie Universit\"at Berlin\\
Institut f\"ur Mathematik\\
Arnimallee 6\\
14195 Berlin\\
Germany}

\date{\vspace{-5ex}}

\begin{document}
\maketitle
\let\thefootnote\relax\footnote{The authors want to thank Robert Scheichl for stimulating discussions \MD{and the unknown referees for their thorough review
and their valuable suggestions that significantly improved the paper.}
This work was supported by the German Ministry for Education and Research (BMBF)
through grant {\em Wear simulation and shape optimization of knee implants. 
Subproject 4: Uncertainty quantification}
and by the German Research Foundation (DFG) through grant
CRC 1114 {\em Scaling Cascades in Complex Systems}, 
Project B01: {\em Fault networks and scaling properties of
deformation accumulation}.}
\addtocounter{footnote}{-1}

\begin{abstract}
While multilevel Monte Carlo (MLMC) methods for the numerical approximation
of partial differential equations with  random  coefficients enjoy great popularity,
combinations with spatial adaptivity seem to be rare. We present an adaptive
MLMC finite element approach based on deterministic adaptive mesh 
refinement for the arising "pathwise"  problems and outline a convergence theory
in terms of desired accuracy and required computational cost.
Our theoretical and heuristic reasoning together with the efficiency
of our new approach are confirmed by numerical experiments.
\end{abstract}

\section{Introduction}

Uncertainty quantification is a  well-established and  flourishing field
in numerical analysis and scientific computing
that connects theoretical challenges with a multitude of practical applications. 
While stochastic Galerkin approaches (cf., e.g., \cite{BabuvskaNobileTempone2010,BabuskaTemponeZouraris2004,XiuKarnadakis2002})
turned out as methods of choice for low dimensional  uncertainties,
Monte Carlo (MC) type of methods prove advantageous  
for high dimensional, highly nonlinear problems. 
While the classical MC method is very robust and extremely simple,
sampling of stochastic data 
entails the numerical solution of numerous deterministic problems 
which makes performance the main weakness of this approach. 
A big step towards efficiency was made by Giles~\cite{giles2008},
who combined MC with multigrid techniques by introducing suitable hierarchies
of subproblems associated with corresponding mesh hierarchies. 
Since then, multilevel Monte Carlo (MLMC) methods became 
a powerful tool in a variety of applications
and its own field of mathematical research.
We refer to elliptic problems
with random coefficients~\cite{elpde1, elpde3,  elpde2, elpde4},
random elliptic problems with multiple scales~\cite{multiscale}, 
parabolic random problems \cite{parab},
random elliptic variational inequalities \cite{stochvi},
and to~\cite{Giles2015} for a detailed overview.

Various approaches have been made to further enhance the efficiency of MLMC.
For a given, quasi-uniform mesh hierarchy,
Collier et al.~\cite{collier2015} and Haji-Ali et al.~\cite{HajiAliNobileSchwerinTempone2016} 
aim at reducing the computational cost of  MLMC by optimizing the 
actual selection of meshes from this hierarchy and other MLMC parameters.

Another, in a sense complementary approach to reduce the required computational cost of MLMC 
is to apply adaptive mesh refinement techniques.
Time discretization of an It\^{o} stochastic differential equation
by an self-adaptively chosen hierarchy of time steps has been suggested by Hoel et al.~\cite{HoelSchwerinSzepessyTempone2014,HoelSchwerinSzepessyTempone2012}
and a similar approach was presented by Gerstner and Heinz~\cite{adapt3},
including applications in computational finance.

Less appears to be known for  partial differential equations with random coefficients.
While a posteriori error estimation and adaptive mesh refinement
have quite a history in finite element approximation of deterministic 
partial differential equations (cf., e.g., \cite{AinsworthOden2011,Verfurth1996}), 
related adaptive concepts for MLMC methods seem to be rare.
Only recently, Eigel et al.~\cite{EigelMerdonNeumann2014} suggested an algorithm 
for constructing an adaptively refined  hierarchy of meshes 
based on expectations of ``pathwise'' local error indicators
and illustrated its properties by numerical experiments.

In this paper, we follow a novel approach,
utilizing a whole family of different pathwise mesh hierarchies 
associated with different MC samples $\omega \in \Omega$. 
More precisely, for a given final tolerance $Tol>0$, 
we choose a sequence of tolerances $Tol_1> \cdots > Tol_L=Tol$
and approximate each of the different pathwise deterministic problems  
arising for each of the different samples $\omega\in \Omega$ on each MLMC level $l$ up to
the accuracy $Tol_l$
by finite elements on a different, adaptively refined ``pathwise'' mesh.
We emphasize that any deterministic refinement strategy can be used for this purpose.
The computation of sample averages is finally performed 
on an inductively constructed global mesh consisting of the union 
of simplices from all these pathwise meshes resulting from the different samples. 

\MD{Based on existing results  
on elliptic variational inequalities~\cite{Gwinner00,KindStamp80} 
and on general MLMC methods~\cite{elpde2,giles2008}, we outline an abstract convergence theory
for adaptive MLMC Galerkin approximations of the expected solution
in an abstract Hilbert space setting.
Error estimates are formulated in terms of the desired accuracy $Tol$ 
and the required computational cost.
Extensions to bounded linear as well as Fr\'echet differentiable functionals
can be obtained from corresponding results in~\cite{elpde3, elpde4}.
The general theory is then applied to MLMC finite element methods.
In the case of uniform refinement
we recover an enhanced version of existing results from~\cite{stochvi} and 
we discuss the assumptions of our abstract  theory 
in light of existing convergent adaptive algorithms for deterministic 
elliptic variational inequalities~\cite{BraessCarstensenHoppe07,SiebertVeeser07}
and optimality results
for linear variational problems~\cite{Binev2004,Carstensen14,Kreuzer2011,Stevenson2007}.
The implementation of the resulting adaptive MLMC finite element methods
is carried out in the software environment {\sc Dune}~\cite{Dune}.
Numerical experiments illustrate our theoretical findings and
the underlying heuristic reasoning.
For problems with highly localized random source term,
we  observe a significant reduction 
of computational cost as compared to uniformly refined meshes.
Optimal bounds for the computational cost are 
observed in all our numerical experiments.
Theoretical justification will be the subject of future research.

The paper is organized as follows.
Section~\ref{sec:varproblem} contains 
the formulation of pathwise elliptic variational inequalities
together with some well-known existence and uniqueness results.
In Section \ref{sec:MMCGM} we present our abstract framework of 
adaptive MLMC Galerkin methods together with error estimates and upper bounds 
for the required computational cost.
These abstract results are  applied to finite element approximations 
in the next Section~\ref{sec:MLMCFEM} 
and numerical experiments are reported in the final Section~\ref{sec:NUM}.
}

\section{A random variational problem} \label{sec:varproblem}
Let $(\Omega, \mathcal{A}, \P)$ be a complete probability space with $\Omega$
denoting a sample space and let $\mathcal{A} \in 2^\Omega$ be
the $\sigma$-algebra of all possible events associated with 
a finite probability measure $\P: \mathcal{A} \to [0,1]$ on $\Omega$.
As usual, $\E[\xi]=\int_\Omega \xi\; d\P$
describes the expectation of a random variable $\xi$ 
and $L^2(\Omega)$ denotes the Hilbert space of square integrable random variables
on~$\Omega$.

For a given separable Hilbert space $H$, 
equipped with the scalar product $(\cdot,\cdot)_H$ 
and the associated norm  $\|\cdot\|_H=(\cdot,\cdot)_H^{1/2}$, 
we introduce the Bochner-type space $L^2(\Omega, \mathcal{A}, \P; H)$ 
of $\P$-measurable mappings $v: \Omega \to H$ with the property 
$\int_\Omega \|v\|^2_H \; d\P(\omega) < \infty$.
We will use the abbreviation $L^2(\Omega; H)=L^2(\Omega, \mathcal{A}, \P; H)$.
It is easily seen that $L^2(\Omega; H)$ is  also a Hilbert space with the 
scalar product 
\[
(v,w)_{L^2(\Omega;H)} = \int_\Omega (v,w)_H \; d\P(\omega), \qquad v,w \in L^2(\Omega;H),
\]
and the associated norm 
$\|  \cdot  \|_{L^2(\Omega;H)}=(\cdot,\cdot)_{L^2(\Omega;H)}^{1/2}$.
The expectation in $L^2(\Omega;H)$ is defined by
\[
\E\left[   v \right] = \int_\Omega v (\omega) \;d \P(\omega)\in H, \qquad v \in L^2(\Omega;H).
\]

Let $a(\omega; \cdot, \cdot)$ and $\ell (\omega; \cdot)$, $\omega \in \Omega$,
denote families of bilinear forms and linear functionals on $H$, respectively.
For a given subset $K\subset H$ and any fixed realization $\omega \in \Omega$,
we consider the ``pathwise'' variational inequality
\begin{equation} \label{path}
u(\omega) \in K: \qquad a(\omega;u(\omega),v-u(\omega))  \geq \ell(\omega;v-u(\omega))   \qquad \forall v \in K.
\end{equation}
Note that in the unconstrained case $K=H$ the inequality \eqref{path} 
can be equivalently rewritten as the variational equality
\begin{equation} \label{pathLIN}
u(\omega) \in H: \qquad a(\omega;u(\omega),v)  =  \ell(\omega;v)   \qquad \forall v \in H.
\end{equation}
\begin{ass} \label{a:VI}
The subset $K$ is non-empty, closed, and convex.
For each realization $\omega \in \Omega$ we have
$\ell (\omega; \cdot)\in H'$  and
$a(\omega; \cdot, \cdot)$ is bounded and coercive in the sense that
\begin{equation}\label{eq:COERCIVE}
 \gamma(\omega)\|v\|_H^2 \leq a(\omega; v,v), \quad 
 a(\omega; v,w)\leq \Gamma(\omega)\|v\|_H \|w\|_H \qquad \forall v,w \in H
\end{equation}
holds with $\gamma(\omega)\geq \gamma_0 >0$ a.e.\ in $\Omega$, and $\Gamma\in L^{\infty}(\Omega)$.
For all fixed $v$, $w\in H$ the mappings $a(\cdot;v,w)$, $\ell(\cdot;v)$
are measurable
and $\ell \in L^2(\Omega;H')$. 
\end{ass}
Assumption~\ref{a:VI} yields existence, uniqueness, and regularity of pathwise solutions
(cf., e.g., \cite[Theorem 2.1]{KindStamp80} and \cite[Proposition 1.2]{Gwinner00}).

\begin{prop} \label{prop:EXIUNIREG}
Let Assumption~\ref{a:VI} hold. Then 
the pathwise problem \eqref{path} admits a
unique solution  for each $\omega \in \Omega$,
the solution map $u: \Omega \mapsto H$ 
is measurable with respect to the Borel $\sigma$-algebra in $H$,
and  $u \in L^2(\Omega;H)$.
\end{prop}

Note that $u \in L^2(\Omega;H)$ implies $\E\left[   u \right]\in H$.
It also follows that 
\[
u\in \cK = \{v \in L^2(\Omega; H)\;|\; v(\omega)\in K \text{ a.e. in }\Omega\} \subset  L^2(\Omega;H)
\]
is the unique solution of the ``mean-square'' variational inequality

\begin{equation} \label{mean}
 u \in \cK:\qquad \E\left[ a(\cdot;u,v-u) \right]  \geq  \E\left[ \ell(\cdot;v-u) \right]   \qquad \forall v \in \cK.
 \end{equation} 

To fix the ideas, we will often concentrate on the bilinear form 
\[
a(\omega; v,w)=\int_D \alpha(x,\omega)\nabla v \cdot \nabla w \; dx
\]
and the functional 
\[
\ell(\omega;v)=\int_D f(x,\omega)\; dx
\]
on the Sobolev space $H=H_0^1(D)$ of weakly differentiable functions 
defined on a Lipschitz domain $D\in \R^d$, $d=1,2,3$, and the subset
\begin{equation} \label{eq:OBSTACLECONSTRAINTS}
K=\{v \in H\;|\; v(x) \geq 0 \text{ a.e. in }D \}.
\end{equation}

Note that random obstacles $\chi\in L^2(\Omega; H_0^1(D))$
can be traced back to the case \eqref{eq:OBSTACLECONSTRAINTS} 
by introducing the new variable $w=u-\chi$.
For a detailed discussion of sufficient conditions on the coefficient $\alpha$ and
the right hand side $f$
for existence and uniqueness of pathwise solutions, 
we refer to Section~\ref{sec:MLMCFEM}.

The remainder of this paper is devoted to the efficient approximation of 
the  expectation  $\E\left[ u \right]$
of the family of pathwise solutions  $u(\omega)$, $\omega \in \Omega$, of \eqref{path}.

\section{\MD{Adaptive} Multilevel Monte Carlo Galerkin  methods} \label{sec:MMCGM}
For given initial tolerance $0 < Tol_1 < 1$ and reduction factor $q<1$
we define a sequence of tolerances by 
\begin{equation} \label{eq:MLTOL}
Tol_l = q Tol_{l-1},\quad l=2,\dots, L,
\end{equation}
with the final desired accuracy $Tol = Tol_L$.
For each $\omega \in \Omega$ we choose an associated  
hierarchy of subspaces $S_l(\omega)\subset H$, i.e.,
\begin{equation} \label{eq:NESTSEQ}
 S_1(\omega) \subset S_2(\omega) \subset \cdots \subset S_L(\omega)\subset H,
\end{equation}
with finite dimensions $N_l(\omega)$ 
and non-empty, closed, convex subsets $K_l(\omega) \subset S_l(\omega)$,
$l = 1,\dots, L$.
We consider the family of pathwise Galerkin approximations
\begin{equation} \label{eq:MULTIGALERKIN}
 u_l(\omega) \in K_l(\omega): \quad a(\omega;u_l(\omega),v-u_l(\omega))  
 \geq \ell(\omega;v-u_l(\omega))  \quad \forall v \in K_l(\omega),\qquad \omega \in \Omega.
\end{equation}

\begin{ass} \label{a:M1}
For all $l=1,\dots, L$
the set-valued map $\Omega \ni \omega \mapsto K_l(\omega) \in H$ is measurable
and there is a $w_l\in L^2(\Omega; H)$ such that
$w_l(\omega)\in K_l(\omega)$ holds for all $\omega \in \Omega$.
\end{ass}

In combination with Assumption~\ref{a:VI}, the Assumption~\ref{a:M1} yields existence,
uniqueness, and regularity of approximate pathwise solutions
(cf., e.g.,  \cite[Theorem 2.3 and 2.7]{GwinnerRaciti06}).

\begin{prop} \label{lem:EXUREG}
Let the Assumptions~\ref{a:VI} and \ref{a:M1} hold.
Then there is a unique solution
$u_l(\omega) \in K_l(\omega)$ of \eqref{eq:MULTIGALERKIN} for each $l=1,\dots,L$ and $\omega \in \Omega$,
the discretized solution map  $u_l: \Omega \mapsto S_l(\omega)\subset H$
is measurable, and $u_l\in L^2(\Omega;H)$.
\end{prop}

Before we approximate the expectation $\E[u]$ 
in terms of (approximations of) $u_l(\omega)$, $\omega \in \Omega$,
let us state some assumptions on  $u_l(\omega)$ and thus 
implicitly on the approximating family of spaces $S_l(\omega)$.

\begin{ass} \label{a:M2}
For all $l=1,\dots,L$ the family $u_l(\omega)$, $\omega \in \Omega$, 
satisfies the discretization error estimate
\begin{equation} \label{eq:PMLACC}
 \|u - u_l\|_{L^2(\Omega;H)} \leq {\textstyle \frac{1}{2\sqrt{2}}} Tol_l.
\end{equation}
\end{ass}

In general, the exact solution $u_l(\omega)$ of  variational inequality \eqref{eq:MULTIGALERKIN} is not available 
but can be only approximated up to a certain tolerance by  an iterative solver.

\begin{ass}\label{a:M3}
For all $l=1,\dots,L$ and each $\omega\in \Omega$, an approximate solution 
$\tilde{u}_l(\omega)\in S_l(\omega)$ 
of the pathwise problem \eqref{eq:MULTIGALERKIN} can be computed with accuracy
\begin{equation} \label{eq:ALGACC}
 \|u_l(\omega)- \tilde{u}_l(\omega)\|_H\leq {\textstyle \frac{1}{2\sqrt{2}}} Tol_l,
\end{equation}
\MD{$\tilde{u}_l: \Omega \mapsto S_l(\omega)\subset H$ is measurable, 
and $\tilde{u}_l\in L^2(\Omega;H)$.}
\end{ass}

Then the expectation $\E(u)$ is approximated by 
the inexact multilevel Monte Carlo Galerkin method
\begin{equation}\label{eq:MLMC}
  \E^L[\tilde{u}_L]= \sum_{l=1}^L \E_{M_l}[\tilde{u}_l - \tilde{u}_{l-1}]
 \end{equation}
with $\tilde{u}_{0}=0$ and suitable $(M_l)\in \N^{L}$.
On each level $l$, we utilize the Monte Carlo approximation 
\begin{equation} \label{eq:MC}
 \E_M[v]= \frac{1}{M}\sum_{i=1}^M v_i(\omega), \qquad M \in \N, \quad v\in L^2(\Omega;H),
\end{equation}
of $\E[v]$ by  independent, identically distributed copies $v_i(\omega)$ of $v(\omega)$, $i=1,\dots,M$.

A basic error estimate  for  Monte Carlo methods is stated in the following lemma.
\begin{lem}   \label{mcestimate}
The Monte Carlo approximation \eqref{eq:MC} of the expectation $\E[v]$ satisfies the error estimate
\begin{equation} \label{eq:SINGLEERR}
\|\E[v]-\E_M[v]\|_{L^2(\Omega; H)} = M^{-1/2} V[v]^{1/2}
\end{equation}
denoting 
\begin{equation} \label{eq:VARN}
V[v]= \E[\| \E[v] - v \|^2_H]  \leq \|v\|^2_{L^2(\Omega; H)}.
\end{equation}
\end{lem}
\begin{proof}
As $v_i(\omega)$ are independent and identically distributed,
we have
\[
\begin{array}{rcl}
\|\E[v] - \E_M[v]\|^2_{L^2(\Omega; H)} 
&=& \displaystyle \E\left[ \left\|\E[v]-\frac{1}{M}\sum_{i=1}^Mv_i(\omega)\right\|^2_H \right] 
= \frac{1}{M^2}\sum_{i=1}^M\E\left[\|\E[v] - v_i(\omega)\|^2_H\right] \\[3mm]
&=& \displaystyle\frac{1}{M} \E \left[\|\E[v] - v\|_H^2\right] 
=  \frac{1}{M} V[v]
\end{array}
\]
and $V[v] =\E[\|v\|^2_H] - \|\E[v]\|^2_H  \leq  \|v\|^2_{L^2(\Omega;H)}$.
\end{proof}

Before we present an error estimate  for the inexact multilevel Monte Carlo method we state a basic identity, that can be proved in a similar way as a related result in~\cite{elpde2}. 

\begin{lem} \label{prop:MLMCCOV}
The inexact multilevel Monte Carlo Galerkin approximation $\E^L[\tilde{u}_L]$ satisfies 
\begin{equation} \label{eq:IDENT}
\|\E[u]- \E^L[\tilde{u}_L]\|^2_{L^2(\Omega;H)} 
= \|\E[u-\tilde{u}_L]\|^2_H  
+ \displaystyle \sum_{l=1}^{L} M_l^{-1} V[\tilde{u}_l - \tilde{u}_{l-1}].
\end{equation}
\end{lem}

\begin{proof}
As Monte Carlo approximations on different levels 
are independent, Lemma~\ref{mcestimate} yields
\[
 \begin{array}{rcl}
\|\E[u]-\E^L[\tilde{u}_L]\|^2_{L^2(\Omega;H)} &=& \|\E[u]-\E[\tilde{u}_L]\|^2_{L^2(\Omega; H)}
+ \|\E[\tilde{u}_L]-\E^L[\tilde{u}_L]\|^2_{L^2(\Omega; H)} \\

&=& \|\E[u-\tilde{u}_L]\|^2_H 
+ \left \|\displaystyle \sum_{l=1}^{L} \E[\tilde{u}_l-\tilde{u}_{l-1}]-\E_{M_l}[\tilde{u}_l-\tilde{u}_{l-1}]\right \|^2_{L^2(\Omega; H)} \\

&=& \|\E[u-\tilde{u}_L]\|^2_H 
+\displaystyle \sum_{l=1}^{L} \|\E[\tilde{u}_l-\tilde{u}_{l-1}]-\E_{M_l}[\tilde{u}_l-\tilde{u}_{l-1}]\|^2_{L^2(\Omega; H)} \\

&=& \|\E[u-\tilde{u}_L]\|^2_H 
+ \displaystyle \sum_{l=1}^{L} M_l^{-1} V[\tilde{u}_l - \tilde{u}_{l-1}].
\end{array}
\]
\end{proof} 

We now prove an error bound for the inexact  multilevel Monte Carlo Galerkin method.
\begin{thm} \label{thm:MLMCERR}
Let the Assumptions~\ref{a:VI} and \ref{a:M1} - \ref{a:M3} hold.
Then the inexact multilevel Monte Carlo Galerkin approximation $\E^L[\tilde{u}_L]$
of the expected value $\E[u]$
satisfies the error estimate
 \begin{equation} \label{eq:MLMCERR}
    \begin{array}{l} 
    \|\E[u]- \E^L[\tilde{u}_L]\|_{L^2(\Omega;H)}^2 \leq\\[3mm]
    \qquad \qquad 
      3 M_1^{-1} ( {\textstyle \frac{1}{4}} Tol_1^2 + V[u] ) +
      {\textstyle \frac{1}{2}} \left( 
      \displaystyle 1 + (1+q^{-1})^2 \sum_{l=2}^{L} M_l ^{-1} q^{2(l-L)} 
     \right)Tol^2.
    \end{array}
 \end{equation}
\end{thm}
\begin{proof}
We estimate  the terms on the right hand side of the identity \eqref{eq:IDENT}.
First we get
\[
\E[\|u-\tilde{u}_L\|_H] \leq \|u-\tilde{u}_L\|_H 
\leq \|u-u_L \|_H + \|u_L-\tilde{u}_L\|_H \leq 2^{-1/2} Tol_L
\]
utilizing the triangle inequality together with Assumptions~\ref{a:M2} and \ref{a:M3}.
Then, for $l=2,\dots,L$ we have
\[
\begin{array}{rcl}
V[\tilde{u}_l - \tilde{u}_{l-1}]
& \leq & \|\tilde{u}_l - \tilde{u}_{l-1}\|_{L^2(\Omega;H)}^2 \\
&\leq &(
\|\tilde{u}_l - u_l\|_{L^2(\Omega;H)} + \|u_l - u\|_{L^2(\Omega;H)}\\
&& \qquad \qquad
+ \|u - u_{l-1}\|_{L^2(\Omega;H)} + \|u_{l-1} - \tilde{u}_{l-1}\|_{L^2(\Omega;H)}
)^2\\
&\leq& {\textstyle \frac{1}{2}} (1+q^{-1})^2 Tol_l^2,
\end{array}
\]
again by  Assumptions~\ref{a:M2}, \ref{a:M3}, and \eqref{eq:MLTOL}.
Finally, for $l=1$, we obtain the  estimate
\[
\begin{array}{rcl}
V[\tilde{u}_1] 
& = & V[(\tilde{u}_1-u_1) + (u_1 - u) + u]\\
& \leq & 3 ( \|\tilde{u}_1-u_1\|^2_{L^2(\Omega;H)} + \|u_1 - u\|^2_{L^2(\Omega;H)} + V[u])\\
& \leq & 3({\textstyle \frac{1}{4}} Tol_1^2 +  V[u]).
\end{array}
\]
Inserting the above  estimates into \eqref{eq:IDENT}, we obtain 
\[
    \begin{array}{l}
    \|\E[u]- \E^L[\tilde{u}_L]\|_{L^2(\Omega;H)}^2 \leq\\
    \qquad \qquad {\textstyle \frac{1}{2}} Tol^2 \displaystyle
    + 3 M_1^{-1} ( {\textstyle \frac{1}{4}} Tol_1^2 + V[u] ) 
    + {\textstyle \frac{1}{2}}(1+q^{-1})^2 \sum_{l=2}^{L} M_l^{-1} Tol_l^2.
    \end{array}
\]
As a consequence of \eqref{eq:MLTOL}, 
we have $Tol_l=q^{l-L}Tol$ and the assertion follows.
\end{proof} 

The error estimate \eqref{eq:MLMCERR} clearly implies 
that the desired accuracy $Tol$ 
is obtained for sufficiently large numbers of samples $M_l$, $l=1,\dots,L$.

\MD{We now investigate the computational cost for the evaluation of $\E^L[\tilde{u}_L]$.
Assuming that the  evaluation of the inexact solution
of the discrete pathwise problems \eqref{eq:MULTIGALERKIN} 
dominates overall work, the computational cost is defined by
\begin{equation} \label{eq:COSTDEF}
\sum_{l=1}^L \sum_{i=1}^{M_l} cost(\tilde{u}_{l,i}(\omega)),
\end{equation}
where $cost(\tilde{u}_{l,i}(\omega))$ stands for the computational cost of one 
evaluation of $\tilde{u}_{l,i}(\omega)$ measured in the number of floating-point operations.
We  relate $cost(\tilde{u}_{l,i}(\omega))$ to the dimension $N_{l,i}(\omega)$ of $S_{l,i}(\omega)$.
}

\begin{ass}\label{a:M4}
\MD{For all $l=1,\dots,L$ and each $\omega\in \Omega$,} an approximation  $\tilde{u}_l(\omega)$  of the  solution $u_l(\omega)$ of \eqref{eq:MULTIGALERKIN} 
 can be evaluated at computational cost bounded by 
\[
c_0(1+ \log (N_l(\omega)))^\mu N_l(\omega)  
\]
 with positive constants $c_0$, $\mu$ independent of $Tol_l$, 
 $N_l(\omega)$, and $\omega \in \Omega$.
\end{ass} 

In order to obtain a bound for the computational cost in terms  of the desired accuracy,  
$Tol_l$ has to be related to $N_l(\omega)$.
\begin{ass} \label{a:M5}
\MD{For all $l=1,\dots,L$ and each $\omega\in \Omega$,} the dimension $ N_l(\omega)$ of the ansatz space $S_l(\omega)$ providing the accuracy \eqref{eq:PMLACC} satisfies 
 \begin{equation}
  N_l(\omega) \leq c_1 Tol_l^{-s},
 \end{equation}
with positive constants $c_1$, $s$ 
independent of  $Tol_l$, $N_l(\omega)$, and $\omega \in \Omega$.
 \end{ass}
 
Now we are ready to state an upper bound for the  computational cost 
for the evaluation of $\E^L[u_L]$  in terms  of the desired accuracy $Tol$.
The proof is carried out along the lines of similar results in~\cite{elpde2,giles2008}.

\begin{thm} \label{thm:MLMCEFFORT}
 Let the Assumptions \ref{a:VI} and \ref{a:M1} - \ref{a:M5} hold. 
 Then there are numbers of samples $M_l$, $l=1,\dots,L$, such that 
 the inexact pathwise multilevel Monte Carlo Galerkin approximation $\E^L[\tilde{u}_L]$
 satisfies the error estimate
\begin{equation} \label{eq:MLMCTOL}
 \|\E[u] - \E^L[\tilde{u}_L]\|_{L^2(\Omega;H)} \leq Tol
\end{equation}
and can be evaluated with computational cost  bounded by
\begin{equation} \label{eq:LMCCOST}
C (1+ |\log(Tol_1)|)^\mu Tol_1^{-s} L^{\mu + c_s} Tol^{-\max\{2,s\}}
\quad \text{with} \quad 
\left \{
\begin{array}{rcl}
 c_s = 2 &for& s=2,\\
 c_s=0   &for& s\neq 2,\\
\end{array}
\right .
\end{equation}
and a constant $C$ only depending  on $c_0$, $c_1$, $q$, $s$, $\mu$, and $V[u]$.
\end{thm}
\begin{proof}
Utilizing Assumptions~\ref{a:M4} and \ref{a:M5},
the computational cost for the evaluation of $\E^L[\tilde{u}_L]$ is bounded by
\[
\begin{array}{l}
\displaystyle 
c_0\sum_{l=1}^L \sum_{i=1}^{M_l}
\left((1+ \log(N_{l,i}(\omega)))^\mu N_{l,i}(\omega)
+ (1+ \log(N_{l-1,i}(\omega)))^\mu N_{l-1,i}(\omega_{li})\right) \\[3mm]
\begin{array}{rcl}
&\leq&    \displaystyle c_0c_1 \sum_{l=1}^L M_l 
\left((1+ \log(c_1Tol_l^{-s}))^\mu Tol_l^{-s} 
+ (1+ \log(c_1Tol_{l-1}^{-s}))^\mu Tol_{l-1}^{-s}\right)\\[3mm]
&\leq& c_0c_1 (1+q^{s})  \displaystyle \sum_{l=1}^L M_l 
\left(1 + \log( c_1 q^{-s(l-1)}Tol_1^{-s} ) \right)^\mu Tol_l^{-s}\\[3mm]
&\leq&   c  L^\mu (1 +  |\log(Tol_1)|)^\mu \displaystyle \sum_{l=1}^L M_l  Tol_l^{-s}
\end{array}
\end{array}
\]
with a constant $c$ depending on $c_0$, $c_1$, $s$, $\mu$, and $q$.
Hence, the desired upper bounds for the computational cost will
follow from corresponding upper bounds for 
\[\sum_{l=1}^L M_l  Tol_l^{-s}.
 \]
We always select $M_1$ to be the smallest integer such that
 \begin{equation} \label{eq:M1}
  M_1 \geq 12 ({\textstyle \frac{1}{4}} Tol_1^2 + V[u])Tol^{-2}, 
\end{equation}
so that the first term in the error estimate \eqref{eq:MLMCERR} is 
bounded by $\frac{1}{4}Tol^2$.
The choice of the other $M_l$, $l=2,\dots,L$, will depend on $s$.

Let us first consider the case $s < 2$. 
We choose the numbers of samples $M_l$
to be the smallest integers such that 
\begin{equation} \label{eq:MS1}
M_l \geq C_1q^{\frac{s+2}{2}(l-1)+2(1-L)}, \quad l=2,\dots,L,
\end{equation}
denoting $C_1=2 (1+q^{-1})^2 (1-q^{\frac{2-s}{2}})^{-1}$.
Inserting any $M_l$, $l=1,\dots,L$, 
with the properties \eqref{eq:M1} and \eqref{eq:MS1} 
into the error estimate \eqref{eq:MLMCERR}, we get
\[
\|\E[u]- \E^L[\tilde{u}_L]\|^2_{L^2(\Omega;H)} 
 \leq {\textstyle \frac{3}{4}}Tol^2
      + {\textstyle \frac{1}{2}}Tol^2
      (1+q^{-1})^2C_1^{-1}\sum_{l=1}^{L-1}q^{\frac{2-s}{2}l} 
   < Tol^2
\]
by exploiting the convergence of geometric series. 
As we have chosen the smallest integers 
with the properties  \eqref{eq:M1} and \eqref{eq:MS1},
we can exploit $2^{2(1-L)}=Tol_1^2 Tol^{-2}$, 
$Tol_l^{-s}= q^{-s(l-1)}Tol_1^{-s}$, $l=2,\dots, L$,
and  similar arguments as above to obtain
\[
\begin{array}{rcl}
\displaystyle \sum_{l=1}^L M_l  Tol_l^{-s} & \leq & 
\displaystyle \left (12 ({\textstyle \frac{1}{4}} Tol_1^2 + V[u])Tol_1^{-s} +
C_1\sum_{l=2}^L q^{\frac{2-s}{2}(l-1)} \right)Tol^{-2}
+ \sum_{l=1}^L Tol_l^{-s} \\[3mm]
& \leq & c Tol_1^{-s} Tol^{-2}
\end{array}
\]
with a positive constant $c$ depending on  $s<2$,  $q$, and $V[u]$.

We now consider other values of $s$. The numbers of samples $M_l$
are chosen to be the smallest integers such that
\begin{equation} \label{eq:MS2}
M_l \geq C_2 L q^{2(l-L)}, \quad l=2,\dots,L,
\end{equation}
with $C_2=2 (1+q^{-1})^2$ for $s = 2$ 
and such that
\begin{equation} \label{eq:MS3}
M_l \geq C_3 q^{\frac{s+2}{2}(l-L)}, \quad l=2,\dots,L,
\end{equation}
with $C_3=2 (1+q^{-1})^2  (1-q^{\frac{s-2}{2}})^{-1}$ for $s > 2$.
The same arguments as above then provide the desired bounds 
for accuracy and computational cost.
\end{proof}

Observe that the logarithmic term in Assumption~\ref{a:M4} 
is reflected by the logarithmic terms  $(1+ |\log(Tol_1)|)^\mu$ and $L^\mu$ 
in the computational cost.

For $L=1$, the approximation $\E[\tilde{u}_L]$
reduces to an inexact version of the classical Monte Carlo method.
Theorem~\ref{thm:MLMCERR} then implies that the error estimate \eqref{eq:MLMCTOL}  holds for 
\[
M \geq {\textstyle \frac{3}{2}} +  6V[u] Tol^{-2}
\]
with $M=M_1$ and $Tol=Tol_1$. The corresponding  computational cost is bounded by 
\[
 C (1+ |\log(Tol)|)^\mu Tol^{-(2+s)}
\]
with $C$ depending on  $c_0$, $c_1$, $s$, $\mu$, $q$, and $V[u]$,
which indicates that, up to initial tolerance and logarithmic terms,
the multilevel Monte Carlo method is by a factor of $Tol^{-\min\{2,s\}}$ faster
than the classical single level version.

\section{Multilevel Monte Carlo Finite Element methods} \label{sec:MLMCFEM}

We consider problem \eqref{path} with the symmetric bilinear form
\begin{equation} \label{eq:COEF}
a(\omega; v,w)=\int_D \alpha(x,\omega)\nabla v(x) \cdot \nabla w(x) \; dx
\end{equation}
and the  linear functional
\begin{equation} \label{eq:RHS}
\ell(\omega;v)=\int_D f(x,\omega) v(x)\; dx,
\end{equation}
both defined on the Sobolev space $H=H_0^1(D)$ of weakly differentiable functions 
on a bounded Lipschitz domain $D\subset \R^d$, $d=1,2,3$, equipped with the norm
\[
\|v\|_H=
\left(\sum_{i=1}^d \textstyle
\|\frac{\partial}{\partial x_i} v \|_{L^2(D)}^2\right)^{1/2} .
\]
The closed convex set $K\in H$ of admissible solutions is given by
\begin{equation} \label{eq:SPACEK}
 \qquad K=\{v \in H\;|\; v(x) \geq 0 \text{ a.e. in } D \}.
\end{equation}

We impose the following assumptions on the random coefficient $\alpha$ and on
the random right hand side~$f$.
\begin{ass}\label{a:coeff} 
The random diffusion coefficient $\alpha$ and the right hand side $f$ 
are strongly measurable mappings $\Omega \ni \omega \mapsto \alpha(\cdot, \omega) \in L^\infty(D)$ and $\Omega \ni \omega \mapsto f(\cdot, \omega) \in L^2(D)$
with the properties
\begin{equation} \label{eq:UNICOER}
0< \alpha_- \leq \alpha(x, \omega) \leq \alpha_+ < \infty 
\quad \text{a.e. in }  D \times \Omega,
\end{equation}
and $f \in L^2(\Omega; L^2(D))$.
\end{ass}

These assumptions imply Assumption~\ref{a:VI}
and thus existence and uniqueness
of pathwise solutions $u(\omega)$ of \eqref{path}
and  $u\in L^2(\Omega; H)$. 
Note that uniform coercivity \eqref{eq:UNICOER} 
can be replaced by weaker conditions (cf., e.g., \cite{ stochvi}).

On the background of the general results from Section~\ref{sec:MMCGM}
we now concentrate on MLMC finite element methods,
for the numerical approximation of the expectation $\E[u]$.
Single level versions are obtained for the special case~$L=1$.

\subsection{Uniform refinement}
We assume for simplicity that $D$ has a polygonal (polyhedral) boundary
and consider the hierarchy of shape regular, conforming, 
quasiuniform partitions $\cT^{(k)}$, $k \in \N$, of $D$ 
into simplices as obtained by successive uniform refinement
of a given, intentionally coarse, initial partition $\cT^{(1)}$
(we will also assume that $\cT^{(1)}$ is sufficiently fine 
in a sense to be specified later).

Then 
\[
h_k=\max_{t \in \cT^{(k)}} \operatorname{diam} (t) = 2^{-k}h_1 ,\qquad k \in \N,
\]
and the associated finite element spaces
\begin{equation} \label{eq:UNIREF}
S^{(k)}=\{ v \in H \; | \; v|_t \text{ is affine } 
\forall t \in \mathcal{T}^{(k)} \},\qquad k \in \N,
\end{equation}
form a hierarchy of subspaces of $H$.
We consider the pathwise approximations
$u^{(k)}(\omega) \in K^{(k)} = S^{(k)}\cap K$ characterized by
\begin{equation} \label{eq:AUXPROB}
 a(\omega;u^{(k)}(\omega),v-u^{(k)}(\omega))\geq \ell(\omega;v-u^{(k)})\quad \forall v \in K^{(k)}, \qquad \omega \in \Omega.
\end{equation}

\begin{ass}\label{a:lipsch}  
The spatial domain $D$ is convex and the random coefficient
$\alpha$ is a measurable map 
$\Omega \ni \omega \mapsto \alpha(\cdot, \omega) \in C^1(\bar D)$ 
with the property $\alpha\in L^\infty(\Omega;C^1(\bar D))$.
\end{ass}

The following discretization error estimate
is a direct consequence of  \cite[Proposition 4.2]{stochvi}.

\begin{thm}\label{pr:MLMCFEER} 
Let the Assumptions~\ref{a:coeff} and \ref{a:lipsch} hold.
Then the  error estimate 
\begin{equation} \label{eq:PMLFEACC}
\MD{\|u-u^{(k)}\|_{L^2(\Omega;H)} \leq C_0 h_k}
\end{equation}
\MD{holds 
with a positive constant $C_0$ that 
is independent of $h_k$, $k \in \mathbb{N}$.}\\
\end{thm}
We make sure that $\cT^{(1)}$ is fine enough to guarantee
\begin{equation} \label{eq:INITUNI}
\MD{\|u-u^{(1)}\|_{L^2(\Omega;H)} \leq {\textstyle \frac{1}{2\sqrt{2}}} Tol_1 }
\end{equation}
by selecting $h_1$ such that  $C_0 h_1 \leq {\textstyle \frac{1}{2\sqrt{2}}}Tol_1$ 
and define a uniform MLMC hierarchy 
in the sense of \eqref{eq:NESTSEQ} according to
\begin{equation} \label{eq:UNISPAC}
S_l(\omega)=S^{(r(l-1)+1)}, \quad K_l=S_l(\omega)\cap K,\quad  l=1,\dots,L,\qquad \omega \in \Omega.
\end{equation}
Then Assumption~\ref{a:M1} is trivially satisfied 
and Theorem~\ref{pr:MLMCFEER} implies the accuracy Assumption~\ref{a:M2}
by choosing  $r \in \N$ such that $2^{-r}\leq q$.
Furthermore, Assumption \ref{a:M3} can be satisfied by sufficiently many steps of
any iterative solver for elliptic variational inequalities that 
converges uniformly in $\omega$
and consists of basic arithmetic or $\max$ operations, 
thus preserving measurability
(cf., e.g., \cite{CottlePang92,GraeserKornhuber09,Kornhuber1994,mandel:1984,MUlbrich2011}).
Then, by Theorem~\ref{thm:MLMCERR}, the resulting uniform, inexact 
MLMC finite element approximation $\E^L[\tilde{u}_L]$ 
with \MD{sufficiently large numbers of MC samples} $M_l$ on each level
satisfies the desired error estimate
\begin{equation} \label{eq:MMCTOL}
\|\E[u] - \E^L[\tilde{u}_L] \|_{L^2(\Omega; H)} \leq Tol.
\end{equation}

It is well-known (cf. \cite[Section 4.5]{stochvi}, \cite[Corollary 4.1]{Badea}) that 
Standard Monotone Multigrid (STDMMG) methods~\cite{Kornhuber1994,mandel:1984}
satisfy Assumption \ref{a:M4} with  $\mu = 4$ in  $d=1$ space dimension,
with $\mu = 5$ in  $d=2$ space dimensions, \MD{and a suitable constant $c_0$}. 
In spite of computational evidence, no theoretical justification 
of mesh-independent convergence rates seem to be available for $d=3$.
Finally, utilizing again Theorem~\ref{pr:MLMCFEER}, we find that
Assumption \ref{a:M5} holds with $s = d$, 
because the dimension $N_l$ of $S_l$ is bounded  by $h_{r(l-1) + 1}^{-d}$ 
and thus by $Tol_l^{-d}$ up to a constant \MD{$c_1$}. 
Hence, Theorem~\ref{thm:MLMCEFFORT} implies the following 
result on the efficiency of uniform MLMC finite element methods.

\begin{thm} \label{eq:MLMCFEEFFORT}
Let the Assumptions~\ref{a:coeff},~\ref{a:lipsch},~\eqref{eq:INITUNI} hold,
and let STDMMG be used for the iterative solution of the pathwise 
discretized obstacle problems of the form \eqref{eq:MULTIGALERKIN}.

\MD{Then there are $M_l$, $l=1,\dots, L$, such that
the resulting  uniform MLMC finite element method
provides an approximation $\E^L[\tilde{u}_L]$
with prescribed accuracy \eqref{eq:MMCTOL} at computational cost bounded by 
\[
C(1 + d |\log Tol_1|)^\mu Tol_1^{-d} L^{\mu + c_s} Tol^{-\max\{2,d\}}
\quad \text{with} \quad 
\left\{ 
\begin{array}{rcl}
 c_s=0,\; \mu = 4 &\text{for}& d=1,\\
 c_s=2,\; \mu = 5 &\text{for}& d=2,
\end{array}
\right .
\]
and a constant $C$ depending only on $c_0$, $c_1$, $q$, and $V[u]$.}
\end{thm}

\MD{In fact, one could chose  
$M_1$ according to \eqref{eq:M1} and $M_l$, $l=2,\dots,L$,
according to \eqref{eq:MS1}  and \eqref{eq:MS2} for  $d=1$ and $d=2$, respectively.}

\MD{The number of refinements in \eqref{eq:UNISPAC} can be defined a priori
for all $\omega \in \Omega$.
Hence, Theorem~\ref{eq:MLMCFEEFFORT} is not new, 
but just a slightly enhanced version, e.g., of Theorem~4.10 from \cite{stochvi}.
Assuming that for all $k\in \N$ and each $\omega \in \Omega$
there is an a posteriori error estimate
$\eta^{(k)}(\omega)$ satisfying
\[
\|u(\omega) - u^{(k)}(\omega)\|_H\leq \eta^{(k)}(\omega),
\]
a priori uniform refinement could be replaced by a posteriori uniform refinement
with possibly different mesh sizes for different $\omega \in \Omega$.
This approach can be regarded as a special case of 
a posteriori adaptive refinement presented in the next subsection.
}

\subsection{Adaptive refinement}
We consider a sequence of nested finite element spaces $S^{(k)}(\omega)$
associated with a corresponding sequence of partitions $\cT^{(k)}(\omega)$, $k\in \cN$,
which, for each fixed $\omega \in \Omega$, 
is obtained by successive adaptive refinement of 
the given fixed initial triangulation $\cT^{(1)}(\omega)=\cT^{(1)}$. 
Let $\cT^{(1)}$ be fine enough to provide the accuracy \eqref{eq:INITUNI}
and we set 
\begin{equation} \label{eq:ADAINIT}
 S_1(\omega)=S^{(1)},\qquad \omega \in \Omega.
\end{equation}

\MD{For each fixed $\omega \in \Omega$ we apply a pathwise adaptive refinement
providing a hierarchy of subspaces $S^{(k)}(\omega)$ 
and corresponding approximations $u^{(k)}(\omega)$.
We  assume  convergence of the pathwise adaptive scheme controlled by 
an a posteriori error estimator.
\begin{ass} \label{ass:MEASADAPT}
For all $k \in \N$ 
and for each fixed $\omega \in \Omega$ we have
\begin{equation} \label{eq:ADPTCONV}
 \|u(\omega) - u^{(k)}(\omega)\|_{H}\leq C_{est} \eta^{(k)}(\omega)\quad \text{and} \quad 
 \eta^{(k)}(\omega)\xrightarrow{k\to \infty} 0
\end{equation}
with an a posteriori error estimator $\eta^{(k)}(\omega)$ and positive constant $C_{est}$ independent of $\omega$.
\end{ass}
}
For each fixed $\omega \in \Omega$, 
there are existing adaptive algorithms 
based on local error indicators 
and corresponding a posteriori error estimates $\eta^{(k)}(\omega)$
that provide convergence \eqref{eq:ADPTCONV}, 
see, e.g., Siebert and Veeser~\cite{SiebertVeeser07},
Braess et al.~\cite[Section 5]{BraessCarstensenHoppe07} or Carstensen~\cite{Carstensen14}. 
\MD{The constant $C_{est}$ in these algorithms usually depends on the initial triangulation $\cT^{(1)}$ and on the ellipticity constants $\alpha_{-}$, $\alpha_{+}$.}

We now define the hierarchy of subspaces \MD{for each $\omega \in \Omega$} according to
\begin{equation} \label{eq:ADAPTSPAC}
S_l(\omega) = S^{(k_l(\omega))}(\omega), \quad l=2,\dots,L,
\end{equation}
where $k_l(\omega)$ is the smallest natural number such that
\begin{equation} \label{eq:MLCRIT}
 \|u(\omega) - u^{(k_l(\omega))}(\omega)\|_H \leq {\textstyle \frac{1}{2\sqrt{2}}}Tol_l
\end{equation}
and $Tol_l$ is chosen according to \eqref{eq:MLTOL}.
This definition makes sense, 
because $k_l(\omega)< \infty$ holds pointwise for each fixed $\omega \in \Omega$
by Assumption~\ref{ass:MEASADAPT}.
Note that 
$k_l(\omega)$ might not be uniformly bounded in $\omega \in \Omega$.
\MD{We assume that adaptive refinement and the accuracy criterion \eqref{eq:MLCRIT} preserve measurability.
\begin{ass}\label{ass:ADAPTMEAS}
For all $l=1,\dots, L$
the set-valued map $\Omega \ni \omega \mapsto S_l(\omega) \in H$ is measurable.
\end{ass}
A rigorous investigation of sufficient conditions for measurability of 
$\omega \to S^{(k)}(\omega)$ and $\omega \to S_l(\omega)$ 
would exceed the scope of this presentation and is therefore postponed to a separate publication.}

\MD{
Assumption~\ref{ass:ADAPTMEAS} clearly implies Assumption~\ref{a:M1} while
the initial condition \eqref{eq:INITUNI} and the accuracy criterion \eqref{eq:MLCRIT}
provide Assumption~\ref{a:M2}.}

Assumption \ref{a:M3} can be satisfied by sufficiently many steps of
any iterative solver for elliptic variational inequalities that 
converges uniformly in $\omega$
and consists of basic arithmetic or $\max$ operations, thus preserving measurability
(cf., e.g., \cite{CottlePang92,GraeserKornhuber09,Kornhuber1994,mandel:1984,MUlbrich2011}).

Like in the uniform case, Assumption \ref{a:M4} can be satisfied by STDMMG methods~\cite{Kornhuber1994,mandel:1984} with  $\mu = 4$ in  $d=1$ space dimension and $\mu = 5$ in  $d=2$ space dimensions with a suitable constant $c_0$.

\MD{Now, instead of the regularity Assumption~\ref{a:lipsch}, we require that  
pathwise adaptive refinement provides quasioptimal meshes uniformly in $\omega\in \Omega$.
\begin{ass}\label{ass:QUASIOPTREF}
For all $l=1,\dots,L$ and each $\omega \in \Omega$, the dimension $N_l(\omega)$
of the finite element space $S_l(\omega)$ defined in \eqref{eq:ADAINIT} and \eqref{eq:ADAPTSPAC}
satisfies
\begin{equation} \label{eq:OPTREF}
N_l(\omega) \leq c_1 Tol_l^{-d}
\end{equation}
with a positive constant $c_1$ independent of 
$Tol_l$, $N_l(\omega)$, and $\omega \in \Omega$.
\end{ass}

For fixed $\omega \in \Omega$ and $K=H$,
the  quasioptimality condition \eqref{eq:OPTREF}
has been established for a variety of adaptive refinement strategies
with a constant $c_1(\omega)$
(cf. e.g., \cite{Binev2004,Stevenson2007,Kreuzer2011}).
Uniform upper bounds for $c_1(\omega)$ as required in Assumption~\ref{ass:QUASIOPTREF}
are observed in the numerical experiments to be presented in the next section. 
Theoretical validation will be the subject of future research.

Now the following convergence result is a direct consequence of Theorem~\ref{thm:MLMCEFFORT}.

\begin{thm} \label{thm_ADAPT}
Let the Assumptions~\ref{a:coeff}, \ref{ass:MEASADAPT} - \ref{ass:QUASIOPTREF}, and~\eqref{eq:INITUNI} hold.
Then there are $M_l$, $l=1,\dots, L$, such that
the adaptive MLMC finite element method 
based on the multilevel hierarchy defined in \eqref{eq:ADAPTSPAC}
provides an approximation $\E[\tilde{u}_L]$ with prescribed accuracy \eqref{eq:MMCTOL}
at computational cost bounded by 
\[
C(1 + d |\log Tol_1|)^\mu Tol_1^{-d} L^{\mu + c_s} Tol^{-\max\{2,d\}}
\quad \text{with} \quad 
\left\{ 
\begin{array}{rcl}
 c_s=0,\; \mu = 4 &\text{for}& d=1,\\
 c_s=2,\; \mu = 5 &\text{for}& d=2,
\end{array}
\right .
\]
and a constant $C$ depending only on $c_0$, $c_1$, $q$ and $V[u]$.
\end{thm}
In fact, one could chose  
$M_1$ according to \eqref{eq:M1} and $M_l$, $l=2,\dots,L$,
according to \eqref{eq:MS1}  and \eqref{eq:MS2} for  $d=1$ and $d=2$, respectively.
}

\section{Numerical Experiments} \label{sec:NUM}
In this section we investigate the adaptive MLMC finite element approach presented 
in the preceding sections from a numerical perspective.
\MD{We use the algorithm proposed by Giles~\cite[Algorithm 1]{Giles2015}
(see also \cite{giles2008}).
Here, the increment of the  number of levels
is associated with uniform mesh refinement for uniform MLMC
and an update of the stopping criterion for adaptive MLMC  to be specified later.
We slightly modified the computation of the optimal number of realizations on each level
by replacing the cost of an individual realization by
the average of the cost of all realizations on the same  level.
In our computations, we used a minimal number $M_{min}$ of samples
setting $M_{min}=100$  for the Poisson problem (cf. Subsection~\ref{subsec:POISSON}) 
and $M_{min}=50$ for the obstacle problem (cf. Subsection~\ref{subsec:OBSTACLE}).} 

The initial accuracy condition \eqref{eq:INITUNI} is addressed by formally setting 
\MD{
\begin{equation} \label{eq:INITOL}
Tol_1 = 2\sqrt{2}C_{est}   \|\err^{(1)}\|_{L^2(\Omega)}.
\end{equation}
with the $L^2(\Omega)$-norm approximated by a Monte-Carlo method with 1000 samples.
We choose  $Tol_l$ according to \eqref{eq:MLTOL} with  $q=\frac{1}{2}$.
The accuracy criterion \eqref{eq:MLCRIT} is replaced by the approximation
\begin{equation}\label{eq:ac_crit}
\err^{(k_l(\omega))}(\omega) \leq { \textstyle \frac{1}{2\sqrt{2}C_{est}}} Tol_l 
  = q^{l-1} \|\err^{(1)}\|_{L^2(\Omega)}
\end{equation}

which is used as stopping criterion on each level in adaptive MLMC.
Note that the unknown constant $C_{est}$ does not appear in our computations.
Both uniform and adaptive MLMC terminate once the stopping criterion
in Giles' algorithm is met.
}

Pathwise adaptive refinement is performed as suggested by Siebert and Veeser~\cite{SiebertVeeser07} with error indicators $\err_t(\omega)$ given by local contributions to the hierarchical error estimator
according to~\cite[Theorem 3.5]{VeeserEtAl}.
Here,  the exact finite element solution is replaced by an approximation
provided by an iterative method to be described below.
In the unconstrained case $K=H$, this approach
is reducing to the classical hierarchical error estimation
(cf., e.g., \cite{bornemann:1996,deuflhard89} or \cite[Section 6.1.4]{weiser}).
Note that the error is estimated in the energy norm.
We use 
D\"orfler marking~\cite{Doerfler96} with $\theta = 0.4$ for the Poisson problem (cf. Subsection~\ref{subsec:POISSON}) 
and $\theta = 0.2$ for the obstacle problem (cf. Subsection~\ref{subsec:OBSTACLE})  together with
local ``red'' mesh refinement~\cite{BankShermanWeiser1983,Bey2000,bornemann:1993}
with hanging nodes~\cite[Section 3.1]{Graeser11}.
Implementation is carried out in the finite element software environment  
{\sc Dune}~\cite{Dune}
involving the \texttt{dune-subgrid} module~\cite{GraeserSander09}
for the evaluation of the sum of different approximate evaluations of $u_{l,i}(\omega)$ on different grids.

Discretized variational inequalities of the form  \eqref{eq:AUXPROB}
are solved iteratively by truncated non-smooth Newton multigrid methods
(TNNMG)~\cite{Graeser11,GraeserKornhuber09} with nested iterations,
because TNNMG is easier to implement and usually converges faster than STDMMG~\cite{Graeser11}. 
Numerical experiments (see, e.g., \cite[Section~5.]{stochvi}) also indicate that TNNMG satisfies Assumption~\ref{a:M4} with $\mu = 0$.
Note that both STDMMG and TNNMG reduce to  classical multigrid with Gau\ss -Seidel smoothing in the unconstrained case  $K=H$.
The accuracy condition \eqref{eq:ALGACC} is replaced by
the uniform stopping criterion
\[
\| u^{(k)}_{\nu + 1} - u^{(k)}_{\nu} \|_H \leq {\textstyle \frac{1}{2\sqrt{2}}} \sigma_{alg} Tol_l
\]
with $u^{(k)}_{\nu}$ denoting the $\nu$-th iterate and
a safety factor $\sigma_{alg}=0.001$ accounting for estimating the algebraic error $\| u^{(k)}- u^{(k)}_{\nu} \|_H$ by $\| u^{(k)}_{\nu + 1} - u^{(k)}_{\nu} \|_H$.
In view of the above mentioned optimal convergence properties of TNNMG, 
\MD{the cost for the evaluation of $\tilde{u}_l(\omega) \in S_l(\omega)$ 
is set to the corresponding number of unknowns $N_l(\omega) = \text{dim }S_l(\omega)$, i.e.,
\[
cost(\tilde{u}_l(\omega)) = N_l(\omega).
\]
In light of \eqref{eq:COSTDEF}, 
the computational cost for the adaptive MLMC method with $L$ levels 
is then given by
\begin{equation}\label{cost}
cost_L = \sum_{l=1}^L\; \sum_{i=1}^{M_l}  N_{l,i}(\omega),
\end{equation}
which  reduces to $cost_L=  \sum_{l=1}^L N_l M_l$ in case of uniform refinement.}


\subsection{Poisson equation with random right-hand side}  \label{subsec:POISSON}

We consider the Poisson problem
\begin{equation} \label{pathAFFIN}
u(\omega)\in \{w \in H^1(D)\;|\; w|_{\partial D}=g(\omega)\}:\quad
a(\omega;u(\omega),v)= \ell(\omega;w) \qquad \forall v \in H_0^1(D)
\end{equation}
with $D=(-1,1)^2$ in $d=2$ space dimensions, the bilinear form 
\begin{equation}
a(\omega;v,w) = \int_D \nabla v\cdot\nabla w \; dx, \quad v,w \in H,
\end{equation}
the right hand side
\begin{equation}
\ell(\omega;v) = \int_Df(x,\omega)v \; dx, \quad v \in H,
\end{equation}
with uncertain source term
\begin{equation}
f(x, \omega) = e^{-\beta |x-Y(\omega)|^2}(4\beta^2|x-Y(\omega)|^2-4\beta),
\end{equation}
and uncertain, inhomogeneous boundary conditions
\[
g(x,\omega) = e^{-\beta|x-Y(\omega)|^2}\quad x \in \partial D.
\]
Here, $\beta$ is a positive constant
and $Y(\omega) = (Y_1(\omega), Y_2(\omega))^T$  is a random vector 
whose components are uniformly distributed random variables 
$Y_1, Y_2 \sim {\mathcal U}(-0.25,0.25)$. 
For each $\omega \in \Omega$ a pathwise solution of \eqref{pathAFFIN}  is given by
\begin{equation}
u(x, \omega) = e^{-\beta|x-Y(\omega)|^2}, \quad x \in D.
\label{laplace_solution}
\end{equation}
As Assumption \ref{a:coeff} is satisfied, this solution is unique and we have spatial regularity
in the sense that $u \in L^2(\Omega; H^2(D))$ (cf.~Assumption~\ref{a:lipsch}).
However, $u(\omega)$ exhibits a peak at  $(Y_1(\omega), Y_2(\omega))\in D$ 
that becomes more pronounced with increasing $\beta$,
thus leading to larger constants  $C_0$ in the uniform error estimate~\eqref{eq:PMLFEACC}.

We will compare the performance of MLMC finite element methods
based on uniform and adaptive refinement, as presented in the preceding Section~\ref{sec:MLMCFEM},
for $\beta = 10$, $50$, $150$.
The initial partition $\cT^{(1)}$ is obtained by applying four uniform refinement steps to
the partition of the unit square $\overline{D}$ into two congruent triangles with right angles at $(1,-1)$ and  $(-1,1)$.

%
Figure~\ref{fig: lapl_er} illustrates the convergence properties of uniform and adaptive MLMC methods for the different values of $\beta$
\MD{by showing the actually achieved error over the inverse of the required tolerance $Tol$.}
Here,  the  error $\|\E[u]-\E^L[\tilde{u}_L]\|_{L^2(\Omega; H^1(D))}$ is approximated by a Monte Carlo method
utilizing $M=5$ independent realizations $\|\E[u]-\E^L[\tilde{u}_L]\|_{H^1(D)}$. 
 For all values of $\beta$, both uniform and adaptive MLMC match the required accuracy $Tol$ 
 as indicated by the dotted line,
 thus nicely confirming our theoretical results (cf.~Theorem~\ref{pr:MLMCFEER} and \ref{thm_ADAPT})
 also in this slightly more general case of random boundary conditions.
 \MD{Due to  limited memory resources 
 the accessible accuracy of uniform MLMC 
 is exceeded by adaptive MLMC for $\beta = 50, 150$}
 \begin{figure}
\centering
\input{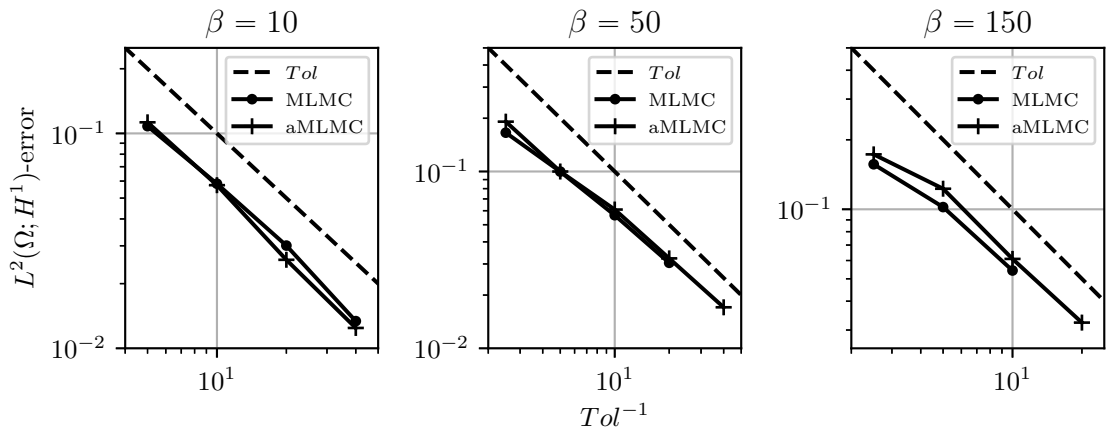}
\caption{Error achieved by uniform and adaptive MLMC over the inverse of
required accuracy $Tol$ for the Poisson problem.}
\label{fig: lapl_er}
\end{figure}
 
\begin{figure}
\centering
\input{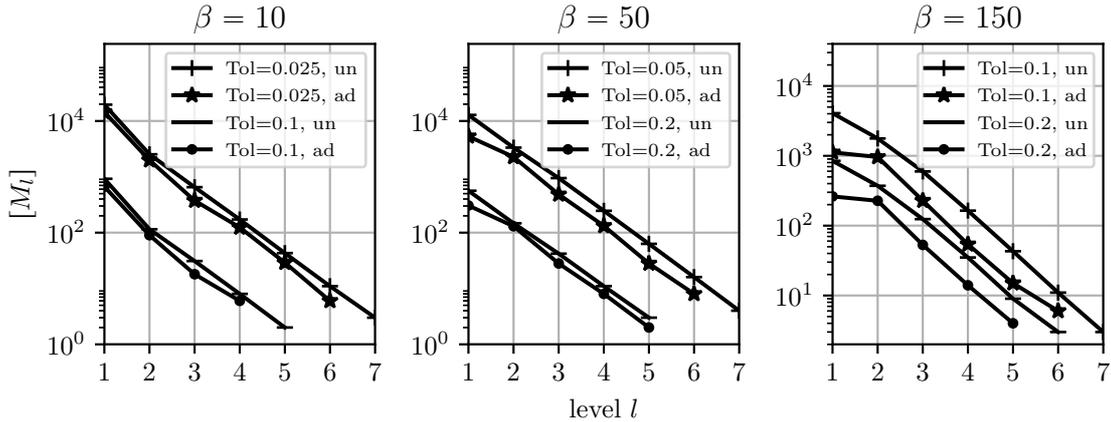}
\caption{Average optimal number of samples over levels  for uniform  (un) and adaptive (ad) MLMC,
different values of $Tol$, and the Poisson problem.}
\label{fig: lapl_samp}
\end{figure}
\MD{We now investigate the corresponding computational effort in terms of required number of samples and mesh size.
Figure~\ref{fig: lapl_samp} shows the average 
numbers of optimal MLMC samples $M_l$ (sometimes smaller than $M_{min}$) over the corresponding  levels $l=1,\dots, L$
for different values of $\beta$ and $Tol$. 
The average is taken over the $M=5$ realizations of  $\|\E[u]-\E^L[\tilde{u}_L]\|_{H^1(D)}$.
It is interesting that the number of samples required for adaptive MLMC is always smaller than for uniform MLMC 
and that the difference becomes larger for larger $\beta$. 
Moreover, adaptive MLMC often  requires less levels than uniform MLMC. } 
\begin{table}[]
\centering
\caption{Average  number of unknowns on different levels for the Poisson problem.
}
\label{tab: dof_lap_10}
\begin{tabular}{|l|rrrrrr|}
\hline
$l$ & 1 & 2 & 3 & 4 & 5 & 6 \\ \hline
uniform & 289 & 1089 & 4225 & 16641 & 66049 & 263169\\
adaptive, $\beta = 10$& 289 & 965 & 3339 & 11719 & 56087 & 218507 \\
adaptive, $\beta = 50$& 289 & 508 & 1017 & 5701 & 16901 & 49895 \\ 
adaptive, $\beta = 150$& 289 & 385 & 929 & 2730 & 6938 & 19606 \\ \hline
\end{tabular}
\end{table}
\MD{Table \ref{tab: dof_lap_10} reports on the average mesh sizes
or, equivalently, the average of the number of the unknowns
$N_{l,i}(\omega)$, $i=1,\dots,M_l$,
on the levels $l=1,\dots,7$ for uniform and adaptive MLMC up to tolerances  $0.025$, $0.05$, and $0.1$ for $\beta=10$, $50$, and $100$, respectively.
Note that adaptive MLMC reached the desired tolerances already on level $L=6$.}
While for $\beta = 10$ the corresponding uniform and adaptive mesh sizes 
 stay relatively close to each other, the mesh sizes for adaptive MLMC 
 for $\beta = 50$, $150$ are considerably smaller than for uniform MLMC.
Even though most of the work in MLMC methods is performed on coarser levels, 
this already indicates a gain of efficiency by adaptive mesh refinement.

\begin{figure}
\centering
\input{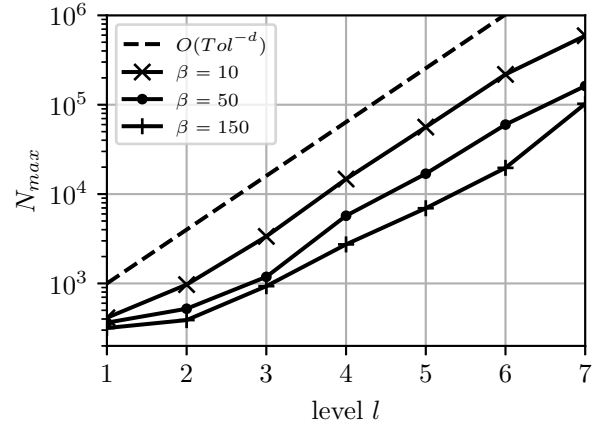}
\caption{Number of  unknowns providing  the accuracy ${\mathcal O}(Tol_l)$  
over  levels $l$ for the Poisson problem.}
\label{fig: lapl_dof}
\end{figure}
Upper bounds of the computational cost of MLMC in terms of the desired accuracy $Tol$
as stated in  Theorem~\ref{thm:MLMCEFFORT}
strongly rely on  Assumption~\ref{a:M5} postulating  $N_l(\omega)={\mathcal O}(Tol_l^{-s})$.
While, under suitable regularity conditions, Assumption \ref{a:M5}  holds with $s=d$ for uniform MLMC,
there is no theoretical evidence yet for adaptive MLMC.
In order to check Assumption \ref{a:M5} for adaptive MLMC numerically, 
\MD{we adaptively computed approximations to  realizations of 
$u_{l,i}(\omega)$, $i=1,\dots,I=1000$, 
up to the tolerance $\frac{1}{2\sqrt{2}C_{est}}Tol_l$ 
according to the stopping criterion \eqref{eq:ac_crit} 
for $l=1,\dots, 7$, and  $\beta = 10$, $50$, $150$.
Figure~\ref{fig: lapl_dof} displays 
the maximal required number of unknowns $N_{l,max}=\max_{i=1,\dots,I}N_{l,i}(\omega)$ over the the number of levels $l=1,\dots,7$.
We observe that $\log(N_{l,max})$ grows like $2\log(q)(l-1)$ (dotted line)
or, equivalently,  $N_{l,max} =  {\mathcal O}(Tol_l^{-2})$ 
for all three values of $\beta$.
This indicates that  adaptive MLMC satisfies Assumption~\ref{a:M5} with $s=d=2$.}

\begin{figure}
\centering
\input{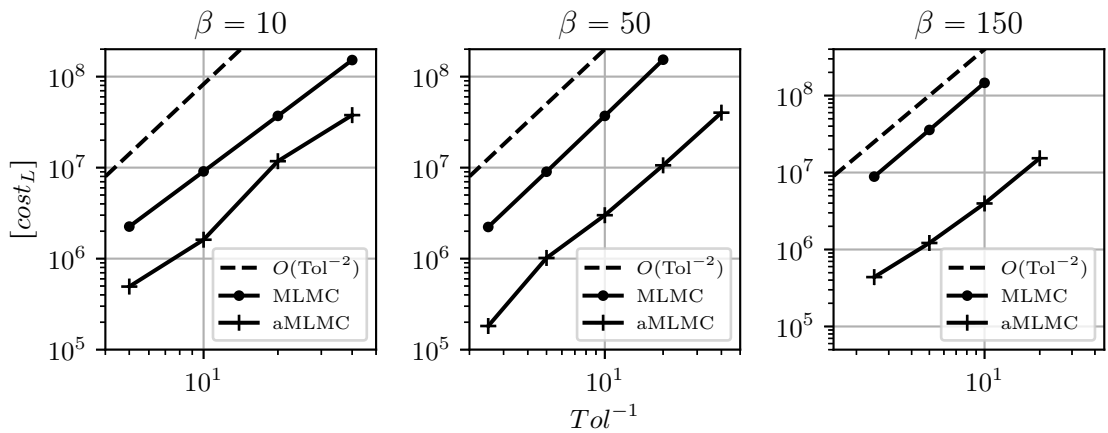}
\caption{Average computational cost of uniform and adaptive MLMC over the 
inverse of required accuracy $Tol$ for the Poisson problem.}
\label{fig: lapl_com}
\end{figure}
On this background, we expect from  Theorem~\ref{thm:MLMCEFFORT} that the computational cost
both of uniform and adaptive MLMC should  asymptotically behave like ${\mathcal O}(Tol^{-2})$.
Figure~\ref{fig: lapl_com} shows the average 
\MD{of $cost_L$, as defined in \eqref{cost}, over the inverse of the required accuracy $Tol$ 
together with the expected asymptotic behavior (dotted line).
As in Figure~\ref{fig: lapl_er}, the average is taken over the $M=5$ realizations of  $\|\E[u]-\E^L[\tilde{u}_L]\|_{H^1(D)}$.
Observe that adaptive MLMC always outperforms uniform MLMC and  the gain is increasing with increasing $\beta$. 
Though the simple model  of  computational cost   \eqref{cost} is frequently used,
it obviously ignores a posteriori error estimation, mesh handling,  interpolation, etc.,
which does occur in adaptive MLMC but not in the uniform case.
We therefore complement our considerations  by a  comparison of the overall run time 
on the machine with 3.3 GHz Intel Xeon E3-1245 processor
with the 7.8 GByte of RAM for different tolerances $Tol$ and different values of $\beta$.
We found that the overall run time  to reach the tolerance, $Tol= 0.025$, $0.05$, and $0.1$
by uniform MLMC was improved by a factor of  $1.1$, $3.2$, and $4.6$  by adaptive MLMC
for  $\beta = 10$, $50$, and $150$, respectively. 
These experiments confirm that uniform MLMC is preferable for sufficiently smooth problems 
while, even without specific software optimization,
adaptive MLMC can substantially reduce the computational cost 
in the presence of random singularities.}


\subsection{Obstacle problem with random  diffusion coefficient and right-hand side } \label{subsec:OBSTACLE}

We consider an elliptic variational inequality of the form \eqref{path} with $D=(0,1)$ in  $d=1$ space dimension,
\[
K= \{v \in H\;|\; v(x) \geq 0 \text{ a.e. in }D \}\subset H, \quad H=H_0^1(D),
\]
the bilinear form 
\begin{equation}
a(\omega;v,w) = \int_D \alpha (x,\omega)\nabla v\cdot\nabla w \; dx, \quad v,w \in H,
\end{equation}
with random  diffusion coefficient
\begin{equation} \label{eq:DIFFCOBS}
\alpha(x,\omega)=1+\frac{\cos x^2}{10}Y_1(\omega) + \frac{\sin x^2}{10} Y_2(\omega),
\end{equation}
and the right hand side
\begin{equation}
\ell(\omega;v) = \int_Df(x,\omega) \; dx, \quad v \in H,
\end{equation}
with random  source term 
\[
f(x,\omega) = 
\left \{
\begin{array}{ll}
\begin{array}{l}
-8 e^{2(Y_1(\omega) + Y_2(\omega))}\left( a(x,\omega)\cdot (3x^2-r^2) \right.
\\[2mm]
\quad \left. + (x^2-r^2)x^2 \left(-\frac{\sin x^2}{10} Y_1(\omega) 
+ 
\frac{\cos x^2}{10}Y_2(\omega)\right)\right)
\;,
\end{array} & x>r \\[6mm]
\begin{array}{l}
4 r^2 e^{2(Y_1(\omega) + Y_2(\omega))} \left(a(x,\omega)\cdot (-1-r^2+x^2) \right.
\\[2mm]
\quad \left. + (-2-2r^2+x^2)x^2\left(-\frac{\sin x^2}{10}Y_1(\omega) 
+
\frac{\cos x^2}{10}Y_2(\omega)\right)  \right) ,
\end{array} & x\leq r
\end{array}
\right .
 \]
denoting 
\[
r=r(Y_1(\omega), Y_2(\omega)) = 0.7+\frac{Y_1(\omega)+ Y_2(\omega)}{10}
 \;.
\]
Here, $Y_1,Y_2\sim {\mathcal U}(-1,1)$ stand for uniformly distributed random variables. 
For each $\omega \in \Omega$ a solution of the corresponding pathwise problem \eqref{path}  is given by
\[
u(x, \omega) = \max{\{(x^2-r^2)e^{Y_1(\omega)+Y_2(\omega)},0\}}^2, \quad x \in D.
\]
As Assumption \ref{a:coeff} is satisfied, this solution is unique and we have $u \in L^2(\Omega;H)$.

We will compare the numerical behavior of  MLMC finite element methods 
with uniform and adaptive spatial mesh refinement
as presented in  Section~\ref{sec:MLMCFEM}.
The initial partition $\cT^{(1)}$  of $\overline{D}=[0,1]$ consists of sixteen closed intervals with length 1/16.

\begin{figure}
\centering
\input{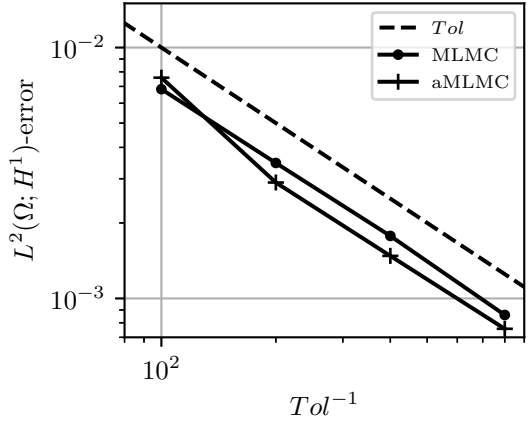}
\caption{Error achieved by uniform and adaptive MLMC over the inverse of required accuracy $Tol$ for the obstacle problem.}
\label{fig: obst_er}
\end{figure}
Figure~\ref{fig: obst_er} shows the error $\|\E[u]-\E^L[\tilde{u}_L]\|_{L^2(\Omega; H^1(D))}$
of uniform and adaptive MLMC over $Tol^{-1}$.
As in the previous numerical experiment,  the exact error $\|\E[u]-\E^L[\tilde{u}_L]\|_{L^2(\Omega; H^1(D))}$ 
is approximated by a Monte Carlo method
utilizing $M=5$ independent realizations $ \|\E[u]-\E^L[\tilde{u}_L]\|_{H^1(D)}$. 
As expected from Theorem~\ref{pr:MLMCFEER} and \ref{thm_ADAPT},
both for uniform and adaptive MLMC the error is bounded 
by the prescribed tolerance $Tol$ indicated by the dotted line.
 Adaptive MLMC appears to be slightly more accurate than the uniform version. 

\begin{figure}
\centering
\input{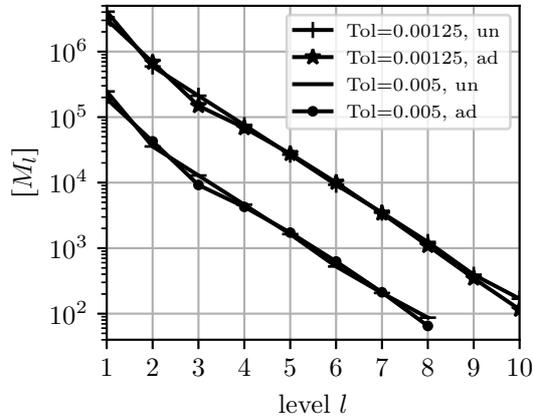}
\caption{Average optimal number of samples  over levels  for uniform  (un) and adaptive (ad) MLMC, different values of $Tol$, and the obstacle problem.}
\label{fig: obst_samp}
\end{figure}
Next, we consider  the  required number of samples and mesh size.
\MD{The average optimal number of  MLMC samples $M_l$
over the corresponding  levels $l=1,\dots, L$ 
are shown in Figure~\ref{fig: obst_samp} for different values of $Tol$.  
Again, the numbers of samples for adaptive MLMC are slightly smaller than for the uniform method. }
\begin{table}[]
\centering
\caption{Average  number of unknowns on different levels for the obstacle problem.}
\label{tab: dof_obs}
\begin{tabular}{|r|rrrrrrrrrr|}
\hline
$l$ & 1 & 2 & 3 & 4 & 5 & 6 & 7 & 8 & 9 & 10\\ \hline
uniform & 17 & 33 & 65 & 129 & 257 & 513 & 1025 & 2049 & 4097 & 8193 \\
adapted & 17 & 19 & 24 & 34 & 57 & 106 & 207 & 443 & 927 & 2150\\ \hline
\end{tabular}
\end{table}
\MD{
The average mesh size  or, equivalently, the average of the number of unknowns
$N_{l,i}(\omega)$, $i=1,\dots,M_l$ on the levels  $l=1,\dots,10$ 
for prescribed tolerance $0.00125$ 
is reported in Table \ref{tab: dof_obs}.
The uniform mesh size on the final level $L=10$  
is about 3.8 times larger than for adaptive MLMC
indicating  the potential of the adaptive approach.}

\begin{figure}
\centering
\input{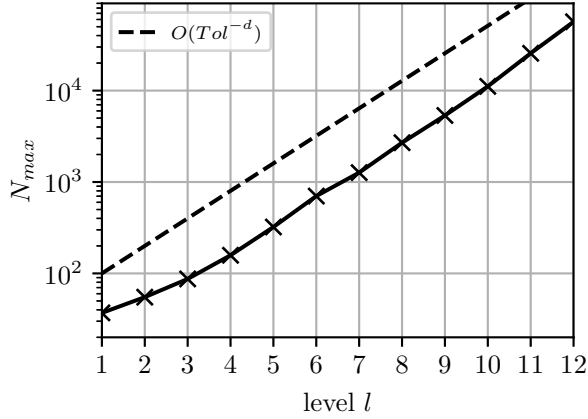}
\caption{Number of  unknowns providing  the accuracy ${\mathcal O}(Tol_l)$  
over  levels $l$ for the obstacle problem.}
\label{fig: obst_dof}
\end{figure}
As the given data clearly satisfy Assumption~\ref{a:lipsch}, the general Assumption~\ref{a:M5}
holds true for uniform MLMC. Hence, Theorem~\ref{eq:MLMCFEEFFORT} provides  the upper bound ${\mathcal O}(Tol^{-2})$ 
for the computational cost of uniform MLMC.
As corresponding theoretical evidence is still missing for adaptive MLMC, we check Assumption \ref{a:M5} numerically.
\MD{To this end, we adaptively computed approximations to  realizations of $u_{l,i}(\omega)$, $i=1,\dots,I=1000$,
up to the tolerance $\frac{1}{2\sqrt{2}C_{est}}Tol_l$ 
according to the stopping criterion \eqref{eq:ac_crit} for $l=1,\dots, 12$.
Figure~\ref{fig: obst_dof} displays $N_{l,\text{max}}=\max_{i=1,\dots,I}N_{l,i}(\omega)$ 
over the number of levels $l$.
We observe that $\log(N_{l,max})$ grows like $\log(q)(l-1)$ (dotted line)
or, equivalently,  $N_{l,max} =  {\mathcal O}(Tol_l^{-1})$ 
indicating that  adaptive MLMC satisfies Assumption~\ref{a:M5}  with $s=d=1$.}

\begin{figure}
\centering
\input{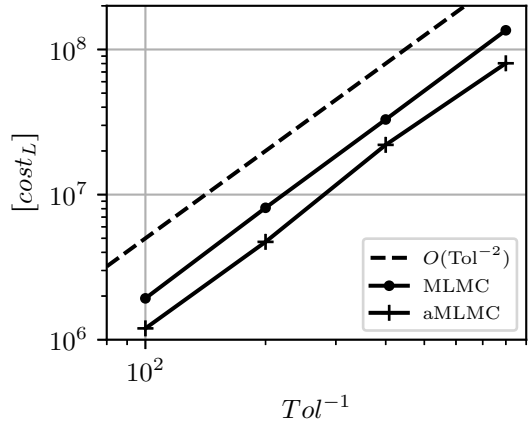}
\caption{Average computational cost of uniform and adaptive MLMC over the
inverse of required accuracy $Tol$ for the obstacle problem.}
\label{fig: obst_com}
\end{figure}
From Theorem~\ref{eq:MLMCFEEFFORT} and Theorem~\ref{thm:MLMCEFFORT},
combined with numerical evidence of Assumption~\ref{a:M5},
we expect that the computational cost both of uniform and adaptive MLMC asymptotically behaves like ${\mathcal O}(Tol^{-2})$.
This is confirmed by \MD{Figure~\ref{fig: obst_com}}  showing 
the  average computational cost over the inverse of \MD{the required  accuracy $Tol$} 
together with the expected asymptotic behavior (dotted line).
\MD{Again, the average is taken over the $M=5$ realizations of  $\|\E[u]-\E^L[\tilde{u}_L]\|_{H^1(D)}$.}
\MD{We observe a gain of efficiency of adaptive MLMC by a factor of 1.75 as compared to the uniform version.}

\MD{We also measured the overall run time on the machine with 
3.3 GHz Intel Xeon E3-1245 processor with the 7.8 GByte of RAM 
for the final tolerance $Tol=0.00125$
and found that (for the given implementation) the overall run time is not improved
by adaptive refinement.}
%
%
%
\bibliography{ref}{}
\bibliographystyle{plain}

\end{document}